\documentclass{amsart}
  \usepackage{amscd,amssymb,epsfig, epsf}
   \usepackage{epic,eepic}

                     \numberwithin{equation}{subsection}

                     \newtheorem{propo}{Proposition}[section]
                     \newtheorem{corol}[propo]{Corollary}
                  \newtheorem{theor}[propo]{Theorem}
                     \newtheorem{lemma}[propo]{Lemma}
                     \theoremstyle{definition}

                     \theoremstyle{remark}

                     \newcommand{\ZZ}{\mathbb{Z}}
                     \newcommand{\RR}{\mathbb{R}}

             \newcommand{\Ker}{\operatorname{Ker}}

              \newcommand{\card}{\operatorname{card}}
              
                     \newcommand{\id}{\operatorname{id}}

\begin{document}
      \title{Axiomatic phylogenetics}
                     \author[Vladimir Turaev]{Vladimir Turaev}
                     \address{%
              Department of Mathematics, \newline
\indent  Indiana University \newline
                     \indent Bloomington IN47405 \newline
                     \indent USA \newline
\indent e-mail: vturaev@yahoo.com} 
                    \begin{abstract}
We use the language of quivers    to formulate a  mathematical framework for  phylogenetics.
\end{abstract}

\maketitle

\section{Introduction}\label{intro}

 Mathematical methods are commonly used in biology which, in some cases,   leads to new mathematical     theories, see   \cite{DHKMS}, \cite{ERSS}, \cite{Gr}, \cite{Re}, \cite{SS}, \cite{Ti}. In this paper we study  certain mathematical ideas suggested by   evolutionary biology.
  Biological evolution   is described in wikipedia as the \lq\lq change in the heritable characteristics of biological populations over successive generations".  Evolution creates a sequence of generations of species,  each generation arising from the previous one via   natural selection and/or genetic drift.
Every  species     produces  descendants   in the next generation.
  The    branching diagram   showing the  evolutionary relationships between   species is called the evolutionary  tree.    This tree has a distinguished vertex - a root - as all life on Earth is believed  to share  a common ancestor known as the last universal common ancestor (LUCA). Phylogenetics studies the evolutionary tree    and aims to recover   it from  the current generation of living organisms. For a review of phylogenetic analysis, see \cite{YR}; for mathematical aspects of phylogenetics, see \cite{SS}, \cite{DHKMS}.

  We  introduce here a mathematical formalism for evolution   emphasizing  its phylogenetic aspects. While it remains to be seen whether this   formalism may be of use in theoretical biology, it does suggest  new mathematical concepts.
 Our main idea   is to consider  not only the historical evolution but   all possible evolutions   of   primitive beings   into   complex organisms. To this end we use the language of quivers  (directed graphs). We introduce analogues of many key notions of phylogenetics in the setting of quivers. This includes  analogues of the notions of species, parents/children, LUCAs,  evolutions, generations, etc. We show that under appropriate assumptions on the quiver, it gives rise to an evolutionary tree. We give   examples of  quivers satisfying these assumptions. 
 



 This work was partially supported by the NSF grant  DMS-1664358.

  \section{Quivers  and evolutions}\label{sect2}


  \subsection{Quivers}\label{sect21}  A  \emph{quiver}\index{quiver} $\mathcal O$ is formed by  a class   of vertices\index{vertex} 
    and a collection of sets $\{\mathcal {O}(A,B)\}_{A,B }$ where $A,B$ run over the vertices  of~$\mathcal O$.  The   elements of the set  $\mathcal {O}(A,B)$ are called     \emph{edges} from~$A$ to~$B$ and   are  represented by   arrows $A \to B$ or $B \leftarrow A$. (We allow  $A=B$, i.e, the quiver~$\mathcal O$ may have loops). To indicate that~$A$ is a vertex  of~$\mathcal O$ we   write $A \in \mathcal O$. A biologically minded reader may replace the word   vertex with \lq\lq species". An arrow $B \leftarrow A$ is understood in the sense that~$B$ is a \lq\lq parent" of~$A$.

  \subsection{Evolutions}\label{sect21+} An \emph{evolutionary chain} or, shorter, an \emph{evolution}\index{evolution}  of length $m\geq 0$    in a  quiver~$ \mathcal O $ is a   sequence
  \begin{equation}\label{evolfifif} A_0  {\longleftarrow} A_1   {\longleftarrow} A_2  {\longleftarrow} \cdots  {\longleftarrow} A_m \end{equation}  where   $A_0, A_1,..., A_m \in \mathcal O$  and   the arrows are    edges   of~$\mathcal O$. We call $A_0$  the      \emph{initial}  vertex     and  $A_m$   the \emph{terminal} vertex  of the evolution. 
  
 For  $A,B\in \mathcal O  $, we write $A \leq B$ and say that~$A$ is an  \emph{ancestor}\index{ancestor}  of~$B$  and~$B$ is a  \emph{descendant}\index{descendant}  of~$A$ if there is an evolution in~$ \mathcal O $  leading from~$A$ to~$B$, i.e.,  starting at~$A$ and terminating at~$B$.
 


   \begin{lemma}\label{lepreorder} The relation   $\leq$   is  reflexive  and transitive.
\end{lemma}

\begin{proof} The reflexivity of  $\leq$ means that $A\leq A$ for any $A \in \mathcal O $.  In other words, every $A\in \mathcal O$ is both its own ancestor and  its own descendant. This is obtained  using  the evolution of~$A$ to itself  of length~0.
The transitivity of  $\leq$ means that if $A \leq B$ and $B \leq C$, then $A \leq C$ for any $A,B,C \in \mathcal O $.  This is obtained by concatenating an evolution   from~$A$ to~$B$ with an evolution  from~$B$ to~$C$. \end{proof}

\subsection{Isotypy}\label{sect22relation} We say that   vertices  $A,B  $ of\index{vertex!isotypic}  a quiver  $\mathcal O  $ are \emph{isotypic}\index{isotypy}  and write $A \sim   B$ if  both $A\leq B$ and $B\leq A$,  i.e., if $A, B$ are descendants of each other.  The vertices $A,B$ are isotypic iff there is a circle formed by arrows and traversing both~$A$ and~$B$:
$$A {\longleftarrow}  \cdots {\longleftarrow} B  {\longleftarrow} \cdots  {\longleftarrow} A.
$$ The   relation~$\sim  $ is an equivalence relation: every vertex  is isotypic to itself (the reflexivity); if $A \sim   B$, then $B \sim   A$ for any   $A,B \in\mathcal O$ (the symmetry); if $   A\sim B \sim C$, then $ A \sim C $ for any   $A,B,C\in\mathcal O$
(the transitivity). The first two properties   follow   from the definitions. The transitivity holds because if  $A\sim B \sim C$, then $A\leq B\leq C$ and $C\leq B\leq A$. Therefore $A\leq C$ and $C\leq A$, i.e., $A\sim C$.

The next lemma shows that isotypic vertices are  equivalent from the evolutionary viewpoint and may be considered as inessential variations of each other. 

   \begin{lemma}\label{lerelation}   For any   $A,B \in \mathcal O  $ the following five conditions are equivalent:

   (i) $A$ and~$B$ are isotypic;

   (ii) $A$ is both an ancestor and a descendant of~$B$;

   (iii) $B$ is both an ancestor and a descendant of~$A$;

   (iv) $A$ and~$B$ have the same ancestors;

   (v) $A$ and~$B$ have the same descendants.
\end{lemma}

\begin{proof}   The equivalences   $ (i) \Longleftrightarrow  (ii)   \Longleftrightarrow  (iii) $ follow directly from the definitions.
If $A \sim B$ and~$X $ is an ancestor of~$A$, then $X\leq A\leq B$. So, $X\leq   B$, i.e., $X$ is an ancestor of~$B$. Similarly,  all ancestors of~$B$ are ancestors of~$A$.
Thus, $(i) \Longrightarrow (iv)$. Conversely, if $(iv)$ holds, then $A$ being its own ancestor must be an ancestor of~$B$, i.e., $A\leq B$. Analogously, $B\leq A$. Thus, $(iv) \Longrightarrow (i)$ and so  $(i) \Longleftrightarrow  (iv) $. The equivalence $ (i) \Longleftrightarrow  (v) $ is checked similarly.  \end{proof}

Note that all vertices    appearing in an evolution between  isotypic vertices are isotypic:
For  any evolution \eqref{evolfifif} between isotypic vertices  $A_0, A_m$, we have $  A_0 \sim A_k$ for $k=0,1,..., m$. 
Indeed, the existence of the  evolution \eqref{evolfifif} implies that $A_0\leq A_k \leq A_m$. Since $A_0, A_m$ are isotypic,    $  A_m \leq A_0$. Therefore   $A_k \leq A_0$. Thus, $A_k \sim A_0$. 

\subsection{Hereditary and anti-hereditary properties} A property, say, $P$ of vertices  of a quiver   is   \emph{hereditary}\index{property!hereditary}   if for any  vertex    having~$P$, all its descendants  also have~$P$.
Similarly, the property~$P$  is     \emph{anti-hereditary}\index{property!anti-hereditary}   if for any  vertex    having~$P$, all its   ancestors also  have~$P$. If the property~$P$ is hereditary  or  anti-hereditary, then it is isotypy invariant, i.e.,  all vertices  isotypic to a vertex  having~$P$  also have~$P$. If a property  is hereditary, then its negation is anti-hereditary and vice versa. 


 \subsection{Examples}\label{sect2example}  1.  Let $\mathcal{SET}$ be the  quiver    formed by      finite non-empty sets as vertices and  maps between  sets as edges.
  An evolutionary chain of length $m\geq 0$ in $\mathcal{SET}$ is   a sequence of  finite non-empty sets $A_0 , A_1,..., A_m$ and maps $\{  A_{k } \to A_{k-1}\}_{k=1}^m$. The elements of    $A_k$  can be viewed as the   species  of the $k$-th generation  while the map  $  A_{k } \to A_{k-1}$ carries  each   species to its parent.  Such evolutions  reflect     asexual reproduction:  each species has a single parent.  (Reproduction involving two    parents  may be formalized     by taking as edges between sets $A,B$  the maps  $   B  \to A \times A$.)
  For any  $A, B \in \mathcal{SET}$,  a   map    carrying~$B$ to a single element of~$A$ yields a length~1 evolution $A\leftarrow B$. Thus, $A\leq B$. Consequently, all vertices  of  $\mathcal{SET}$ are  isotypic.

2.   Let~$\mathcal{S}$ be the  quiver      formed by      finite non-empty sets as vertices and  surjective maps between sets as edges.
 It is clear  that   $A \leq B$ for $A,B\in   \mathcal S  $ if and only if $\card(A)\leq \card(B)$ where $\card$ is the number of elements of a set. Therefore, $A \sim   B$ if and only if $\card(A)=\card(B) $. 



 \section{Primitivity and the height}\label{sect3newstyly}


\subsection{Primitive vertices }\label{sect22}   The role of   LUCAs  in evolutionary biology will be played  here by so-called primitive vertices. We call a vertex~$A$   of a quiver~$ \mathcal O   $         \emph{primitive}\index{vertex!primitive}  if all   ancestors  of~$A$ also are descendants of~$A$. Thus, $A \in   \mathcal O$ is primitive if $B\leq A \Longrightarrow A \leq B$ for all $B\in   \mathcal O  $. It is clear from our definitions that a vertex~$A$ is primitive iff all its ancestors are isotypic to~$A$.  

A quiver  may have no primitive vertices.  If it has primitive vertices,  then   they may be non-isotypic to each other. 
 At the same time,  all vertices  isotypic to a primitive vertex    are primitive, as is clear   from  the following lemma.

  \begin{lemma}\label{le1} All ancestors of a primitive vertex  are primitive.
\end{lemma}

\begin{proof}  Let~$A  $ be   a primitive  vertex  of a quiver~$\mathcal O$ and let $B\in \mathcal O$ be an ancestor of~$A$. If~$C $ is an ancestor of~$ B$, then $C\leq B \leq A$ and by transitivity, $C\leq  A$. Since~$A$ is primitive, we must have  $A \leq C$. Since $B \leq A$, the transitivity yields $B \leq C$. Thus,
$C\leq B \Longrightarrow B \leq C$, i.e., $B$ is primitive. \end{proof}
 


\subsection{The height}\label{sect22hhfull} 
 By a \emph{full evolution}\index{evolution!full}    for a vertex $X\in\mathcal O$ we mean an evolutionary chain in~$\mathcal O$ which starts at a primitive vertex  and terminates at~$X$. We view such an evolution as an evolutionary history of~$X$. A full evolution for~$X$ does not necessarily exist, and if it exists, it is not necessarily unique.

 We define the \emph{height}\index{height} $ h(X) $ of  $X\in \mathcal O$   to be the smallest   integer $h\geq 0 $ such that there is a full evolution for~$X$ of length~$h$. The height measures the  evolutionary complexity of~$X$, i.e., the number of  steps needed  to evolve from a primitive vertex to~$X$.  If there are no full  evolutions for~$X$, then   we set $h(X)= \infty$.

 It is clear that a vertex $X\in \mathcal O$  has a finite height if and only if~$X$ has a primitive ancestor. Concatenating   evolutionary chains we  conclude that the property of a vertex  to have finite height is hereditary.
  Note also that $h(X)=0$ if and only if~$X$ is   primitive, and   $h(X)=1$ if and only if~$X$ is not primitive but there is an  edge from~$X$ to a primitive vertex.

A full evolution  for  a vertex $X\in \mathcal O $   is said to be     \emph{short}\index{evolution!short}   if its length is  the smallest among all full evolutions for~$X$. Such an evolution    exists     if and only if $h(X)<\infty$ and its length is equal to $h(X)$.  

 \begin{lemma}\label{le2bbb}   Let $ \alpha=(  A_0 \leftarrow    \cdots \leftarrow A_m) $ be a full evolution
   for a vertex  $  A_m\in \mathcal O $. Then     $h(A_k) \leq k$ for  all $k=0,1,..., m$. If $\alpha$ is short, then $h(A_k) = k$ for   all~$k$.
\end{lemma}

\begin{proof}    For   $k=0,1,..., m  $, let
       $\alpha_k=(A_0  \leftarrow    \cdots  \leftarrow  A_k)$ be the initial  segment of~$\alpha$ of length~$k$
       and let $\alpha^{ k}=(A_k  \leftarrow    \cdots  \leftarrow  A_m)$ be the terminal  segment of~$\alpha$ of length $m-k$. Since $\alpha$ is a   full evolution,   $A_0$ is a primitive vertex  and so $\alpha_k$ is a full evolution for $A_k$. Thus, $h(A_k) \leq k$.
  If there is a full evolution  for $A_k$ of length $<k$  then  concatenating it with $  \alpha^{ k} $ we obtain  a full evolution
  for $A_m $ of length   $<k+(m-k)=m$. Therefore, if $\alpha$ is   short, then so is $\alpha_k$ and $h(A_k)=k$.
   \end{proof}

  Lemma~\ref{le2bbb} implies that   any vertex    of finite  height $m\geq 1$ has    ancestors of heights    $ 0,1,..., m-1$.
  
   \subsection{Examples}\label{sect2exampleexamplea---} 1.  All vertices  of the quiver  $\mathcal{SET}$ from Example~\ref{sect2example}.1  are primitive and have  zero height.   Consider next the  quiver~$\mathcal{S}$   from Example~\ref{sect2example}.2. A finite non-empty set $X\in \mathcal S$  is primitive   if and only if $\card(X)=1$.   Every   $X\in   \mathcal S  $ is the terminal vertex   of a full evolution:   if $\card (X)=1 $, then this  is the length~0 evolution   $(X)$;   if $\card (X) \geq 2$, then this is  the length~1 evolution $A \leftarrow X$ where~$A$ is a  1-point set  and the arrow stands for the only map  from $X$ to~$A$. Thus, $h(X)=0$ if $\card(X)=1$ and $h(X)=1$ otherwise.  

 2.    Let $\Gamma$ be a rooted tree, i.e., a connected graph without cycles and with a distinguished vertex~$\ast$ (the root).   We direct all edges of~$\Gamma$ as follows:  if an edge of~$\Gamma$ connects   vertices $A,B$ and there is a path from~$\ast$   to~$A$ missing~$B$, then this edge is directed from~$B$ to~$A$. This turns~$\Gamma$ into a quiver.  It is easy to check that:  different vertices    of~$\Gamma$   cannot be isotypic; the vertex~$\ast$ is the only primitive vertex   of~$\Gamma $;   every  vertex  $X\in   \Gamma  $ is the terminal vertex   of a unique full evolution  formed by the vertices and edges of the shortest path from $\ast$ to~$X$,  this   full evolution  is short.
 

  3. Consider the quiver with  three vertices $A,B, C$ and three edges  leading   from~$B$ to~$A$, from~$B$ to~$C$, and from~$C$ to~$B$.
  Then $A \leq B \sim C$   and   $h(A)=0$, $h(B)=1$, $h(C)=2$.  Thus,   isotypic vertices    may have different heights.
  
  4. The quiver with vertices  $\{A_k\}_{k\in \ZZ}$ and   arrows $\{A_k \leftarrow A_{k+1}\}_{k\in \ZZ}$ has no primitive vertices    and $h(A_k)=\infty$ for all~$k$.
  

 \section{Universal evolutions and phylogenetic vertices}\label{sect3newstylycrit}

 We introduce   so-called universal evolutions and phylogenetic vertices.

  \subsection{Universal evolutions}  Given  two evolutionary chains  $$ \alpha=(A_0 \leftarrow   \cdots \leftarrow A_m)\quad {\text {and}} \quad \beta=( B_0 \leftarrow      \cdots \leftarrow   B_n)  $$ in a quiver~$\mathcal O$,   we say that~$ \alpha $    \emph{embeds} in~$\beta  $
    if $m\leq n$ and
    there are integers $$0\leq r_0< r_1< \cdots < r_m\leq n$$ such that $A_k\sim B_{r_k}$ for   $k=0,1,..., m$. For $m=n$, this condition amounts to   $A_k\sim B_k$ for all~$k$. In this case we say  that  the evolutions  $\alpha, \beta$  are \emph{isotypic}.

             A \emph{universal evolution}\index{evolution!universal}   for a vertex  $X \in \mathcal O$    is a full evolution for~$X$   which  embeds in all   full evolutions for~$X$. Thus, the vertices    of a universal evolution for~$X$ must appear (in the same order) in any full evolution for~$X$, at least up to isotypy.  A universal evolution for~$X$ is necessarily short. So, its length is equal to $h(X)$.  A universal evolution for~$X$ may exist only if $h(X) <\infty$. 

   Clearly,  any two universal evolutions   for a vertex  $X \in \mathcal O$     are isotypic.                 Any evolution isotypic to a universal evolution is itself universal. 
        We view  a universal evolution for~$X$   as  a complete evolutionary history of~$X$.

  \subsection{Phylogenetic  vertices }\label{sect22hh} A vertex $X\in \mathcal O$    is \emph{phylogenetic} if there is a universal evolution  for~$X$  in~$\mathcal O$.
   For example, for  a  primitive~$X $,  the length zero evolution $(X)$ is universal. Thus, all primitive vertices  are phylogenetic.

   We     state several    properties of phylogenetic vertices.

  \begin{theor}\label{le2}   Let $  X\in \mathcal O$ be a phylogenetic vertex. Then:
  
  (i)  $X$ has   a  primitive ancestor;
  
    (ii) all primitive ancestors  of~$X$ are   isotypic to each other;
    
     (iii)  all short  full evolutions    for~$X$   are universal;
     
     (iv) all  vertices  appearing in   a universal evolution for~$X$ are phylogenetic. 

\end{theor}

\begin{proof} Pick a universal evolution   $   \alpha=(A_0 \leftarrow   \cdots \leftarrow A_m ) $  for $X=A_m$. Clearly, $A_0$ is a primitive ancestor of~$X$ which yields (i). To prove (ii), consider another  primitive ancestor~$B$ of~$X$ and  an evolution
 $ \beta=(B_0 \leftarrow   \cdots \leftarrow B_n)$  from $B_0=B$ to $B_n=X$. Since $B$ is primitive, $\beta$ is  full. Since~$  \alpha$  is universal, it  embeds in~$ \beta$. Consequently, $A_0 \sim B_k$ for some $k  \leq n$. Then  $B=B_0\leq B_k\leq A_0$. Hence $B\leq A_0$, and the primitivity of $A_0$ ensures that $B  \sim A_0$. This gives (ii).
 
  To prove (iii), consider a  short full evolution $ \beta=(B_0 \leftarrow \cdots \leftarrow B_m)$   for $X=B_m$. Since~$\alpha$   embeds  in~$\beta$ and $\alpha, \beta$ have the same length,
  $A_k\sim B_{k}$ for all $k\leq m$. Therefore~$ \beta$ is universal.
 
For $k=0,1,..., m$,
let $\alpha_k$ and $\alpha^k$ be the  initial and terminal segments of the universal evolution~$\alpha$ as defined in the proof of Lemma~\ref{le2bbb}.
    Since $\alpha$ is a full evolution for~$A$,  $\alpha_k$ is a full evolution for $A_k$.
 Any full evolution~$\gamma$ for $A_k$   concatenated with $\alpha^k$  yields   a full evolution
$  \gamma \, \alpha^k$ for $A_m $. Since the evolution   $ \alpha_{k} \, \alpha^{k }=\alpha$ is universal, it embeds in
$\gamma \, \alpha^k$. Therefore
$  \alpha_{k } $ embeds in~$\gamma $. Thus,   $\alpha_{k }$ is a universal evolution for $A_k$ and    the vertex $A_k$ is phylogenetic.
  \end{proof}

 \subsection{Examples}\label{secteklflf} The full evolutions     in Examples~\ref{sect2exampleexamplea---}.1 and~~\ref{sect2exampleexamplea---}.2  are universal. All vertices  of the quivers  in these examples are phylogenetic.





  \section{Critical ancestors and normality}\label{Critical ancestors and normality}
  
We   further develop our language  and, in particular, introduce the notion of a critical ancestor. In this section,  $\mathcal O$ is an arbitrary quiver.
  
   \subsection{The  step inequality}\label{sectmonotonestep} The height of a vertex  of~$\mathcal O$  cannot increase too quickly under evolutions. In fact, for any length~1 evolution (i.e., an edge  of  our quiver) $ A \leftarrow B $ we   have the \emph{step  inequality}   \begin{equation}\label{onestep} h(B) \leq h(A)+1.\end{equation}    More precisely,  if $h(A)<\infty$, then \eqref{onestep} holds and, in particular, $h(B)<\infty$. Indeed,  concatenating a   full evolution for~$A$ of length $h(A)$ with the evolution $ A \leftarrow B $ we  obtain a full evolution for $B$ of length $h(A)+1$. If $h(A)=\infty$, then the inequality \eqref{onestep} provides no information on $h(B)$.

\subsection{Critical ancestors}\label{sectcritttmoncriteee}   A vertex  $A\in \mathcal O$   is a  \emph{critical ancestor}\index{ancestor!critical}  of a vertex $B\in \mathcal O$ 
    if $h(A) <\infty$ and there is an evolution $A \leftarrow A_1 \leftarrow \cdots \leftarrow B$ of length $\geq 1$  such that $h(A_1)=h(A)+1$. We view  critical ancestors of~$B$   as   gatekeepers  yielding  access to higher  levels in evolutions   from the primitives to~$B$. It follows from   Lemma~\ref{le2bbb}   that any vertex  of a finite height $m\geq 1$ has critical ancestors of heights $0,1,..., m-1$.  
  Primitive vertices  have no critical ancestors because  all ancestors of a primitive vertex    are primitive and have zero height. So, the equality $h(A_1)=h(A)+1$ above cannot hold.

A critical ancestor of a vertex  $B \in \mathcal O$ is necessarily a critical ancestor of all descendants of~$B$. Consequently, isotypic vertices  have the same critical ancestors. 
 

\subsection{Normal vertices }\label{Normal vertices }  
 We call a  vertex  $B \in \mathcal O$    \emph{normal}\index{vertex!normal} if   any  two critical ancestors of~$B$ of the same height are isotypic.  
  For instance,  all  primitive vertices  are normal simply because they have no critical ancestors.

\begin{lemma}\label{corodddSDl1fgfllla}\label{BEST}  The normality of a vertex  is anti-hereditary. Vertices  isotypic to a normal vertex     are normal.
\end{lemma}

\begin{proof}   If~$B$ is a descendant of~$A$, then all critical ancestors of~$A$ are   critical ancestors of~$B$.  Consequently, if $B$ is normal, then so is~$A$. The second claim of the lemma  follows from the first claim. \end{proof}

 \begin{theor}\label{corodddSDl1fgfllla+}\label{BEST+}    Any  normal vertex   of finite height is phylogenetic.
\end{theor}

\begin{proof}  Let~$X$ be a   normal vertex  of finite height~$m $. If $m=0$, then~$X$ is primitive and therefore phylogenetic. Assume   that $m\geq 1$ and pick a short full evolution   $\alpha =(A_0 \leftarrow     \cdots \leftarrow A_m )$ for $X=A_m$.  We will prove  that~$\alpha$  is universal, i.e.,  that~$\alpha$ embeds in any full evolution $\beta=
(B_0 \leftarrow     \cdots \leftarrow B_n)$ for $X=B_n$. Since~$\beta$ starts at a vertex   of zero height   and terminates at a vertex  of height~$ m$, the  step inequality   implies that for each $k=0,1,..., m-1$, there is   $r \in \{0,1,..., n-1\}$ such that $h(B_r)=k$ and $h(B_{r+1})=k+1$.  Let $  r_k $  be the smallest such~$r$.   Also, set $r_m=n$.  Since  the segment $B_0 \leftarrow     \cdots \leftarrow B_{r_k}$ of~$\beta$ starts from a vertex  of zero height   and terminates at a vertex  of height~$k$,  the same argument as above shows that  if $k>0$, then there is an index $r <  r_k  $ such that $h(B_r)=k-1$ and $h(B_{r+1})=k$. On the other hand,   $r_{k-1}$ is the smallest such index. Therefore $r_{k-1} \leq r  <r_k$.  Hence, 
$  r_0< r_1< \cdots < r_m $.  We claim that   $A_k\sim B_{r_k}$ for   $k=0,1,..., m$.  For $k=m$, this    is obvious because $A_m=X=B_n=B_{r_m}$.  For    $k<m$,  Lemma~\ref{le2bbb} and the definition of $r_k$ imply that    both $A_k$ and $B_{r_k}$ are critical ancestors of~$X$ of height~$k$.    By the normality of~$X$,  we have  $A_k\sim B_{r_k}$. Therefore,  $\alpha$ embeds in~$\beta$.
\end{proof}



  \subsection{Examples}\label{sect2exampleexamplea---dfdfdf}  In Example~\ref{sect2exampleexamplea---}.3, the vertex~$A$ is primitive and has no critical ancestors. The vertex~$B$ has $A, B$ as critical ancestors. This is clear from the evolutions
  $A \leftarrow B$ and $B \leftarrow C \leftarrow B$. The vertex~$C$ is not a critical ancestor of~$B$ because this quiver does  not have a vertex with height $h(C)+1=3$. 
  The vertex~$C$ is isotypic to~$B$ and therefore has the same critical ancestors $A, B$.  
  In this example and    in Examples~\ref{sect2example}.1,   \ref{sect2example}.2,~\ref{sect2exampleexamplea---}.2 all vertices are normal.   The quiver     in  Example~\ref{sect2exampleexamplea---}.4 has no normal vertices  and its vertices    have no critical ancestors.

   \section{Monotonous   and phylogenetic quivers }\label{Monotonous quivers  and their properties}\label{sect2BISBBB}

 We  introduce monotonous     and       phylogenetic quivers.

   \subsection{Monotonous  quivers }\label{sectmonot}  The height of a vertex  may decrease  under certain evolutions, and we view such   evolutions  as degenerate.   We call a   quiver~$\mathcal O$    \emph{monotonous}\index{quiver!monotonous} if it does not have such degenerate evolutions, i.e., if  for any edge  $A\to B$ in~$\mathcal O$, we have $h(A) \geq h(B)$. This condition may be reformulated by saying that the  descendants  of any  vertex  $B\in \mathcal O$ have the height $\geq  h(B)$. A useful consequence:  in a monotonous quiver, isotypic vertices     have the same height. For   monotonous quivers  we can invert  Theorem~\ref{BEST+} as follows.

  \begin{theor}\label{le2++++dbnnnbnb}   A vertex  of  a monotonous quiver  is phylogenetic if and only if it is normal and has finite height.
\end{theor}

\begin{proof} In view of Theorem~\ref{corodddSDl1fgfllla+}, it suffices to prove the \lq\lq only if" part. Consider a phylogenetic vertex~$X$ of  a monotonous quiver  and   a universal evolution  $  \alpha= ( A_0 \leftarrow   \cdots \leftarrow A_m  )$ for $X=A_m$ where $m=h(X)<\infty$. The evolution~$\alpha$ is short, and so, by Lemma~\ref{le2bbb}, the vertices  $A_0, A_1, ..., A_{m-1}$ are critical ancestors of~$X$ of   heights  respectively $0,1,..., m-1$. To prove the normality of~$X$,  consider an arbitrary critical ancestor~$B$  of~$X$ of height $r< \infty$. By the definition of a critical ancestor, there is an evolution $\beta=(B \leftarrow B_1 \leftarrow \cdots \leftarrow X)$     with $h(B_1)=r+1$. By the monotonicity,  all vertices  in~$\beta$ except~$B$ have the height $\geq h(B_1)>r$.
  Pick any   short full evolution~$\gamma$  for $B$. By Lemma~\ref{le2bbb}, all vertices    in $\gamma$ except~$B $ have   height $< r$.  Concatenating $\gamma$ with~$\beta$, we obtain a   full evolution $ \gamma \beta$ for~$X$ whose only vertex  of height~$r$   is~$B$. Since the universal evolution $\alpha$  must embed in $ \gamma \beta $, we   have $A_{r} \sim B $. Thus, $X$ is normal. 
\end{proof}

It is clear that in a monotonous quiver, the property of a vertex  to have finite height is anti-hereditary. Combining this observation with Lemma~\ref{BEST} and Theorem~\ref{le2++++dbnnnbnb}, we obtain the following.

  \begin{corol}\label{le2++++c12}   In a monotonous quiver, the   phylogeneticity of a vertex   is anti-hereditary. 
\end{corol}

 Theorem~\ref{le2++++dbnnnbnb} and Lemma~\ref{le2bbb}  imply  the following claim.

   \begin{corol}\label{le777what} In a monotonous quiver,   a phylogenetic vertex~$X$ has precisely $h(X)$ isotypy classes  of critical  ancestors.
\end{corol}


  \begin{corol}\label{le777oldest} In a monotonous quiver,      universal evolutions for isotypic  phylogenetic vertices   are themselves isotypic.
\end{corol}

\begin{proof}    Consider  universal evolutions $   A_0 \leftarrow   \cdots \leftarrow A_m  $     and    $  B_0 \leftarrow   \cdots \leftarrow B_m  $ for   isotypic  phylogenetic vertices  $X=A_m,Y=B_m$ of height~$m$.    For  any $k=0,1,..., m-1$, the   critical ancestors $A_k$ of $X$ and $B_k$  of~$Y$ have the same height~$k$. Since $X \sim Y$, the vertex $B_k$ is also a critical ancestor of~$X$. By Theorem~\ref{le2++++dbnnnbnb}, $X$ is normal and therefore   $A_k \sim  B_k$. Also,
$A_m=X \sim Y=   B_m$.  \end{proof}


\subsection{Phylogenetic  quivers }  A quiver    is   \emph{phylogenetic}\index{quiver!phylogenetic} if   it is monotonous  and all its vertices  are phylogenetic.     Theorem~\ref{le2++++dbnnnbnb} shows  that  a quiver  is phylogenetic  if and only if it is  monotonous  and all its vertices  are  normal and have finite height.

It is easy to check that the quivers  $\mathcal{SET}$ and    $\mathcal{S}$  from Section~\ref{sect2example}   as well as the tree quiver  ${\mathcal O} (\Gamma)$   from Section~\ref{sect2exampleexamplea---}.2 are phylogenetic.  Further  examples of phylogenetic quivers  can be derived from  various algebraic theories   involving   filtrations. As a specific case, we consider  a quiver  of finite nilpotent groups. (Similar   phylogenetic quivers  can be formed from finite solvable groups and from nilpotent/solvable  finite-dimensional Lie algebras.) Recall the   lower central series $ G_0 \supset G_1   \supset \cdots$ of a group~$G$: by definition, $G_0=G$ and  for $n\geq 0$, the group $G_{n+1}\subset G_n$ is   generated by the commutators $ xyx^{-1} y^{-1}$ with $x\in G_n$ and $y\in G$. The group~$G$ is \emph{nilpotent} if $G_n=\{1\}$ for some $n \geq 0$, and the smallest such~$n$ is  
denoted   $n(G)$. Let~$\mathcal{N}$ be the  quiver   formed by      finite nilpotent groups   and    group epimorphisms $f:G\to H$ such that $ n(G)\geq 1$ and $\Ker (f) \subset G_{n(G)-1}$. It is easy to check that:    groups in~$\mathcal{N}$ are isotypic   if and only if they are isomorphic; a group in~$\mathcal{N}$ is primitive as a vertex    if and only if it is trivial;  for any $G \in \mathcal{N}$,
 the sequence of   quotient groups and projections
 $$\{e\}=G/G_0 \longleftarrow G/G_1 \longleftarrow \cdots \longleftarrow G/G_{n(G)}=G$$
  is a universal evolution for~$G$.   The quiver~$\mathcal {N} $
  is      phylogenetic.

\subsection{Remark} 
Any  monotonous quiver~$\mathcal O$ determines a  quiver~$\mathcal O'$ consisting of all phylogenetic vertices  of~$\mathcal O$ and all edges  between them in~$\mathcal O$.   It is easy to show using Corollary~\ref{le2++++c12}  that the quiver~$\mathcal O'$ is  phylogenetic.



 \section{The evolutionary sequence and the evolutionary forest}\label{The evolutionary sequence and $E$-sequences}

\subsection{The evolutionary sequence}  Here we restrict ourselves to so-called small quivers.   We call   a quiver      \emph{small}\index{quiver!small} if  the isotypy classes     of its vertices     form a set. This  condition  is satisfied in all  our examples.

  Consider a small phylogenetic quiver~$\mathcal O$  and let $\widetilde {\mathcal {O}   }$ be  the  set of isotypy classes of vertices  of~$\mathcal O$.    Each    vertex    $A\in  \mathcal O$ represents an element   $[A] $ of~$ \widetilde {\mathcal {O}   }$.  Two vertices  $A, B \in  \mathcal O$ represent the same element of $\widetilde {\mathcal {O}   }$ if and only if $A \sim B$. The relation~$ \leq$ in the class of vertices  of~$\mathcal {O}$ induces a relation   $ {\,\leq \,}$  in~$\widetilde {\mathcal {O}   }$:   for   $a,b\in \widetilde {\mathcal {O}   }$, we set $a{\,\leq \,} b$   if $A\leq B$ for some
(and then for all)  $A, B \in\mathcal O$ representing respectively~$a$ and~$b$.   The relation  ${\,\leq \,}$ in   $\widetilde {\mathcal {O}   }$ is a partial order, i.e.,  it is reflexive, transitive, and antisymmetric  (if $a{\,\leq \,} b$ and $b{\,\leq \,} a$, then $a=b$).
 Clearly, $\widetilde {\mathcal {O}   }=\amalg_{m\geq 0}\,   \mathcal {O}_m$ where $\mathcal {O}_m  $ is the set of  isotypy classes of    vertices   of~$\mathcal O$ of height~$m$. In particular,  $ \mathcal {O}_{0}$ is the set of isotypy classes of primitive vertices  of $ \mathcal {O}$.

 For each $m\geq 1$, we define the \emph{parental map}\index{parental map}
\begin{equation}\label{parent}  p=p_m: \mathcal {O}_m \to \mathcal {O}_{m-1} \end{equation}
as follows.
 For any   $A \in \mathcal {O} $ of height $m\geq 1$, consider a universal evolution
 $   A_0 \leftarrow   \cdots \leftarrow A_{m-1} \leftarrow A_m= A $  and set $p ([A])= [A_{m-1}] $.  By Corollary~\ref{le777oldest},  this yields a well-defined map \eqref{parent}.
The resulting sequence of sets and maps
\begin{equation}\label{elseq}  \mathcal {O}_{0}  \stackrel{p} {\longleftarrow} \mathcal {O}_{1}  \stackrel{p} {\longleftarrow} \mathcal {O}_{2}  \stackrel{p} {\longleftarrow} \mathcal {O}_{3} \stackrel{p} {\longleftarrow} \cdots   \end{equation}
  is called the \emph{evolutionary sequence}\index{evolutionary sequence} of~$\mathcal {O}$.

\begin{theor}\label{aboutpreceq} Let $a \in \mathcal {O}_m$ and $b \in \mathcal {O}_n$ with $m,n \geq 0$. Then:

(i)   $a{\,\leq \,} b$ if and only if  $m\leq n$ and $a {\,\leq \,} p^{n-m} (b) \in \mathcal {O}_m$;

(ii) For $m=n=0$, we have  $a{\,\leq \,} b \Longleftrightarrow a=b$;

   (iii)  For $m=n \geq 1$, if $a{\,\leq \,} b$, then $p(a)=p(b) \in \mathcal {O}_{m-1}$;
   
   (iv) For $m=n-1$, the equality $ a=p(b)$ holds if and only if  there is an edge   $B\to A $ in~${\mathcal {O}}$ such that   $A,B\in \mathcal O$ represent respectively the isotypy classes $a, b$. 
\end{theor}

Theorem~\ref{aboutpreceq} is proved   below. Note   that by Claim~(i), the partial order   in $\widetilde {\mathcal {O}   } $ is fully  determined by its restrictions to the sets $\{{\mathcal O}_m\}_m$ and the parental maps.

\subsection{Proof of Theorem~\ref{aboutpreceq}} We begin with a lemma.

\begin{lemma}\label{aboutmonooo} Let   $    \alpha=(A_0 \leftarrow   \cdots \leftarrow A_m) $ and $ \beta=(  B_0 \leftarrow   \cdots   \leftarrow B_n    )$ be    universal evolutions for  vertices  $ A_m,  B_n$ in  a monotonous quiver  such that  $A_m \leq B_n$. If $m=n$, then $A_{m-1} \sim B_{n-1}$. If   $m<n   $, then $  A_{m}\leq B_{n-1}$.
\end{lemma}

\begin{proof}  
By Lemma~\ref{le2bbb},  $h(A_k)=k$  for all $k\leq m$ and $h(B_l)=l$ for all $l\leq n$. Since
 $A_m \leq B_n$,   there is an evolution $\gamma $ from~$A_m $   to~$B_n$, and then $\alpha \gamma$ is an  evolution
  from $A_0$ to~$B_n$.  Since~$\beta$ is universal, it embeds in $\alpha \gamma$. So,   $B_{n-1}$   is isotypic to a vertex  in   $\alpha\gamma$  of the same height $ n-1$ (here   we use the monotonicity of the quiver).  Clearly,  all vertices  in~$\gamma$ have  height $\geq  h(A_m)=m $. If   $m=n$, then  the only vertex   of height $ n-1=m-1$ in $\alpha\gamma$ is $A_{m-1}$. Thus, $A_{m-1} \sim B_{n-1}$.
If   $m<n$, then  all vertices  in~$\alpha$ other than $A_m$ have height $ <m \leq  n-1$. In this case,    $B_{n-1}$   has to be isotypic to a vertex,~$C$, appearing  in the evolution~$ \gamma$, and so   $A_{m}\leq C \leq B_{n-1}  $ and $A_{m} \leq B_{n-1}  $. \end{proof}

We now prove Theorem~\ref{aboutpreceq}.  We start with Claim (i). Suppose   that     $a {\,\leq \,} b$. By the monotonicity, $m \leq n$. If $m=n$, then    $a \leq b=p^{n-m}(b) $. If $m<n$, then the second claim of Lemma~\ref{aboutmonooo} implies that $a {\,\leq \,} p(b)$. Iterating, we get   $a {\,\leq \,} p^{n-m} (b)$. Conversely, suppose that $m \leq n$ and   $a \leq p^{n-m}(b)$. It  follows from the definitions that $p (b) {\,\leq \,} b$ for all~$b$. Hence,   $$a {\,\leq \,}  p^{n-m} (b) {\,\leq \,} p^{n-m-1}  (b) {\,\leq \,}  \cdots {\,\leq \,} p (b) {\,\leq \,} b  .$$

Claim (ii) holds because any vertex   of zero height   is primitive   and so is isotypic to all its ancestors. Claim~(iii)   follows from the first claim of  Lemma~\ref{aboutmonooo}.  We  prove Claim (iv). If $ p(b)=a$, then pick a  representative $B\in {\mathcal O} $ of~$b$ and  a  universal evolution
 $   B_0 \leftarrow   \cdots \leftarrow B_{n-1} \leftarrow B_n= B$. The   edge  $B_{n-1} \leftarrow B_n$   satisfies  our conditions because   $ [B_{n-1}] =p(b)=a$.  Conversely, suppose that there is an edge   $  A \leftarrow B $ in~${\mathcal {O}}$ such that   $A,B  $  represent respectively $a,  b$. Concatenating a short full evolution
   for~$A$    with the 1-edge  evolution $ A \leftarrow B$, we obtain a full evolution $     \cdots \leftarrow A \leftarrow  B $  of length $h(A)+1=m+1=n= h(B) $. This evolution is short and, by Theorem~\ref{le2}(iii), universal. Therefore $p(b)= [A]=a$.

 \subsection{The   evolutionary forest}\label{sectevolgraph} Any sequence  of sets and maps
\begin{equation*}\label{elseqarb}  P_{0}  \stackrel{p} {\longleftarrow} P_{1}  \stackrel{p} {\longleftarrow} P_{2}  \stackrel{p} {\longleftarrow} P_{3}  \stackrel{p} {\longleftarrow}  \cdots   \end{equation*}
determines a graph~$\Gamma$: take the disjoint union  $\amalg_{m\geq 0}\,  P_m$ as the set of vertices  and connect
each    $a \in P_{m }$ with $m\geq 1$     to   $ p(a)\in P_{m-1}$ by an edge.  This graph  is a forest in the sense that all its components are trees. Clearly,  every component  of~$\Gamma$  has a unique vertex  in $ P_0$. If   $\card (P_0)=1$,
 then $\Gamma $ is a tree. In this case we  define   a metric~$d$ in the set $\amalg_{m\geq 0}\,  P_m$: For  any  elements $a,b$ of this set,   the distance   $d(a,b)$ is the minimal number of edges in a path in  $\Gamma $ from~$a$ to~$b$.
Clearly, $d(a,b)=k+l$ where $k,l \geq 0 $ are minimal  integers  such that $p^k(a)=p^l(b)$. 
Note  that for each $m\geq 0$, the restriction of~$d$ to the set $P_m $ takes only even values and is an ultrametric (the definition of an ultrametric is recalled in Section~\ref{sect30}).

Applying  these  constructions  to the evolutionary  sequence  \eqref{elseq}, we obtain a   forest   called  the  \emph{evolutionary forest}\index{evolutionary forest} of the  quiver~$\mathcal {O} $.
If  $\card ({\mathcal {O}}_0)=1$, i.e., if all  primitive vertices  of~$\mathcal {O}$ are isotypic, 
 then we also obtain a metric  in   the set~$\widetilde {\mathcal {O}   }$ of isotypy classes of vertices  of~$\mathcal O$.

 \subsection{Examples}\label{sect2exampleexampleaopdfts}   The evolutionary forest of the quiver  $\mathcal{SET}$     is a single point.  The evolutionary forest  of the quiver  $\mathcal{S}$ from  Example~\ref{sect2example}.2  is  a wedge of a countable number of segments. 
The evolutionary forest  of the
 tree quiver   $  \Gamma $ from Example~\ref{sect2exampleexamplea---}.2  can be identified with~$\Gamma$.

\section{$E$-sequences  and reconstruction}\label{Timing and reconstruction}

\subsection{$E$-sequences}\label{$E$-sequences}
Axiomatizing the properties of   the evolutionary sequences, we define so-called $E$-sequences.  An   \emph{$E$-sequence} consists   of partially ordered sets $(P_m, {\,\leq \,})_{m\geq 0}  $   and maps
$(p=p_m:P_m\to P_{m-1} )_{m\geq 1}  $
such that the partial order  in $P_0$ is trivial and   for any   $a, b\in P_m$ with  $m\geq 1$, if    $a{\,\leq \,} b$, then $p(a)=p(b)$. We will sometimes  use the strict partial order $<$ in  $P_m$ defined   by $a<b$ if $a\leq b$ and $a\neq b$.
  Two $E$-sequences  $P, P'$
are \emph{isomorphic}  if there are  bijections $\{f_m:P_m\to P'_m\}_{m \geq 0}$ such that for all $m\geq 1$, we have $p' f_m=f_{m-1} p:P_m \to P'_{m-1}$   and $f_m$ is order-preserving in the sense that $a{\,\leq \,} b \Longleftrightarrow f_m(a) {\,\leq \,} f_m(b)$ for any $a,b\in P_m$.

 By Theorem~\ref{aboutpreceq}, the evolutionary sequence of a  small phylogenetic  quiver  is an $E$-sequence. We   show now  that   all  $E$-sequences arise in this way.

  \begin{theor}\label{le777esche} Every $E$-sequence is isomorphic to  the evolutionary sequence of a  small phylogenetic quiver.
\end{theor}

 \begin{proof} Given an $E$-sequence $P$, we define a  quiver~$\mathcal O $ as follows. The   vertices  of~$\mathcal O$ are the elements of the set $\amalg_{m\geq 0}\,  P_m$. For  all $a,b \in P_m$ with $m\geq 0$,  there is a  single edge     $a\to b$ if $b < a$  and a  single edge    $a \to p (a) \in P_{m-1}$ if $m\geq 1$. We claim that the quiver~$\mathcal O$ is phylogenetic.  A typical evolutionary chain    in~$\mathcal O$   starts  at some   $a_0 \in  P_{m }$ with $m  \geq 0$ and    consecutively takes  bigger and bigger elements $a_0< a_1<  \cdots  $ of   $P_m$. At some step, one either stops at $a_i\in P_m$ or  proceeds to the higher level by  taking for $a_{i+1}$ any   element of the set $p^{-1}(a_i) \subset P_{m+1}$ (if this set is non-void). Then   the whole process is repeated starting at $a_{i+1}$, etc. (One has to stop eventually.) The antisymmetry of   the  partial order   implies that  an evolution in~$\mathcal O$ of non-zero length  cannot start and end at the same vertex. Consequently, isotypic vertices  of~$\mathcal O$  must coincide so that~$\mathcal O$ is small and  $\widetilde{\mathcal O} = \amalg_{m\geq 0}\,  P_m$.  The primitive vertices  of~$\mathcal O$ are   the elements of   $P_0$ (here we use    that the partial order  in $P_0$ is trivial).   All full evolutions for any $a \in P_m$ with $m\geq 1$   start at $p^m(a)\in P_0$ and  include the vertices 
 $\{p^k(a)\}_{k=1}^m $.  Therefore the full  evolution \begin{equation*}\label{uniuni} p^m(a) \leftarrow p^{m-1} (a) \leftarrow \cdots \leftarrow p(a) \leftarrow a  \end{equation*}
    is universal for~$a $.   As a consequence,     $h(a)=m$ and ${\mathcal O}_m=P_m$. Therefore the quiver~$\mathcal O$ is monotonous and all its vertices  are phylogenetic. It is clear that   the evolutionary sequence of~$\mathcal O$ is isomorphic to~$P$.
      \end{proof}

\subsection{Reconstruction}\label{Reconstruction} The  main objective  of phylogenetics is to recover the   evolutionary tree    from  the current generation of species. We  briefly discuss reconstruction in  our context. The  reconstruction aims to   recover the   initial segment   
\begin{equation}\label{elseqppppp}  P_{0}  \stackrel{p} {\longleftarrow} P_{1}  \stackrel{p} {\longleftarrow} \cdots     \stackrel{p} {\longleftarrow} P_N  \end{equation}
of an  $E$-sequence  from the set $P_N$, eventually endowed with additional data.
 For simplicity, we   assume here that   $\card(P_0)=1$ and all the  maps $p$ are surjective.   One well-known approach to   reconstruction   uses  the  ultrametric~$\rho$ in $P_N$ defined by letting   the distance $\rho(a,b)$ between any    $a,b\in P_N  $   to be    the minimal    integer $k  \geq 0$   such that $p^k(a)=p^k(b)$. (Note that $\rho(a,b)=\frac{1}{2}d(a,b)$ where~$d$ is the metric defined in Section~\ref{sectevolgraph}.) 
  The   sets and maps~\eqref{elseqppppp} can be fully recovered from  the ultrametric space $(P_N, \rho)$. Namely, for $s=0, 1,...,N$, the elements of $P_s$ can be identified with   balls  in $ P_N$ or radius $N-s$; for $s \geq 1$, the map     $p :P_s\to P_{s-1}$   carries    a  ball $B \subset P_N$ or radius $N-s$  to the unique ball  of radius $N-s+1$ in $P_N$ containing~$B$. 
   Next, we encode in terms of $P_N$ the given strict partial order $<$   in $P_1,..., P_N$. This partial order induces a binary relation $\prec$ in $P_N$  by the rule $a \prec  b$ if   $a\neq b$  and $p^{k-1}(a) < p^{k-1}(b)$ for  $k=\rho(a,b)\geq 1$. Conversely, the  strict partial order $<$ in  $P_1,..., P_N$   can be fully recovered from  $\prec$ and~$\rho$: two  balls  $B, B'\subset P_N$  of the same radius~$ r$  satisfy $B  < B'$ if and only if    $a \prec b$ and $\rho(a,b)=  r+1$ for some (and then for all) $a\in B , b\in B'$. In particular. for $r=0$, two   points $a,b\in P_N$ satisfy $a< b$ if and only if $a\prec  b $ and $  \rho(a,b)=1$.

\subsection{Remark} So far we have studied evolutions in a static world in which all vertices  (species) coexist together.  To relate  to the real world, we briefly discuss the timeline.
One   way to involve time is to accept the following three  principles:

 (i) (the moment zero) all primitive vertices    come  to existence at the same moment of time, the moment zero;

 (ii) (the molecular clock) the  time needed for an accomplishment of an  evolutionary chain is equal to   a   constant   coefficient~$C$ times the length of the chain;

  (iii) (the least wait) every non-primitive vertex   comes to existence at the earliest possibility, i.e., at the end of a short  full evolution.

These   principles ensure that each  vertex~$X$ of a phylogenetic quiver~$\mathcal O$ evolves in the moment of time $C \cdot h(X)$. The sequence \eqref{elseq} is then the sequence of generations: each set $\mathcal O_m$ with $m\geq 0$ is the generation of   vertices    that have evolved at the moment of time $Cm$.

The principles (i) and (ii) above can be generalized by  agreeing that (i)$'$ each   primitive vertex comes to existence at a certain moment of time (not necessarily the same) and (ii)$'$ each edge carries a   positive   length and the  length of any evolution is the sum of the lengths of the constituent edges.

\subsection{Remark} A  binary relation $\prec$ in an ultrametric space $(X, \rho)$ arises as in Section~\ref{Reconstruction}  from the initial segment of length $N\geq 1$ of an $E$-sequence iff

(i) the ultrametric $\rho$ takes   values in the set $\{0,1,..., N\}$;

   (ii) $a\prec b \Longrightarrow b \nprec a$ for all $a,b \in X$ (in particular, $a\nprec a$ for all $a  \in X$);

   (iii) for any   distinct $a,b,c \in X$,

   - if $a\prec b  $ and $\rho(a,c)< \rho(a,b)$, then $c \prec b$;

- if $a\prec b  $ and $\rho(b,c)< \rho(a,b)$, then $a \prec c$;

- if $a\prec b  \prec c$ and $\rho(a,b)= \rho(a,c)=\rho(b,c)$, then $a \prec c$.

%


 \section{Clades in phylogenetic quivers }\label{sectclade}

\subsection{Clades}\label{sect30clades} Any vertex   $A $ of a   quiver~$\mathcal O$ determines a   quiver  ${\mathcal O}_A$ formed by all descendants of~$A$ and all edges between them  in~$\mathcal O$. Following the standard taxonomic terminology, we   call ${\mathcal O}_A$ the \emph{clade}\index{clade} of~$A$. Clearly, $A\in {\mathcal O}_A$. We   state   a few  properties of  ${\mathcal O}_A$.

\begin{lemma}\label{thththt2++++rreee} The  primitive vertices  of   ${\mathcal O}_A$ are all the vertices  of~$\mathcal O$ isotypic to~$A$.   In particular, $A$ is a primitive vertex  of ${\mathcal O}_A$.
\end{lemma}

\begin{proof}  If    $B \in {\mathcal O}_A$  is primitive in ${\mathcal O}_A$, then  the relation $A \leq B$ implies that $B \leq A$, i.e., that $A ,B $ are isotypic in~$\mathcal O$.  Conversely, if  $B\in \mathcal O$ is isotypic to~$A$, then $B\in   {\mathcal O}_A$ and for any $C\in   {\mathcal O}_A$, we have $   B \leq A \leq C$ so that $B \leq C$.
Thus, $B$   is primitive as a vertex  of ${\mathcal O}_A$.  
\end{proof}

\begin{lemma}\label{thththt2++++rree++e}    All vertices  of ${\mathcal O}_A$ have   finite height in ${\mathcal O}_A$.
\end{lemma}

\begin{proof}  
Since $A$ is primitive in ${\mathcal O}_A$ and all vertices  of $ {\mathcal O}_A$ are terminal vertices  of evolutions starting at~$A$, all vertices  of $ {\mathcal O}_A$ have final height.
\end{proof}

The heights of vertices  in the quivers    $  {\mathcal O} $  and  $  {\mathcal O}_A$ will be  denoted respectively   by~$h$ and  $h_A $. The following theorem   estimates $h_A$ via $h$  for monotonous~$  {\mathcal O} $. 
 
 \begin{theor}\label{ththdbdnht2++++eee} If~$A$ is a vertex  of finite height in a    monotonous quiver~${\mathcal O}$, then  for all    $B\in {\mathcal O}_A$ we have $\infty > h(B) \geq h(A)$ and 
$h_A(B) \geq h(B)-h(A)  $.
\end{theor}

\begin{proof} The inequalities $\infty > h(B)  \geq h(A)$  follow  from the definitions of Section~\ref{sect22hhfull}  and  the monotonicity of~$\mathcal O$.    By Lemmas~\ref{thththt2++++rreee} and~\ref{thththt2++++rree++e}, there is a short full evolution,  $\beta$, in~${\mathcal O}_A$ starting at a vertex  $C\in {\mathcal O}$ isotypic to~$A$ and terminating at~$B$. We have $h(C)=h(A) <\infty$ so that there is a short full evolution, $\gamma$, for~$C$ in~$\mathcal O$. Then $\gamma \beta$ is a full evolution for~$B$ in~${\mathcal O}$ of length $  h(C)+h_A(B) =h(A)   +h_A(B) $.
 Therefore $  h(A)+h_A(B) \geq h(B)$  and  $ h_A(B) \geq h(B)-h(A)$.   \end{proof}

 \subsection{Regular vertices }  We   call  a vertex~$A$ of a quiver~$  {\mathcal O}$    \emph{regular} if   for any   $B  \in  {\mathcal O}_A$ with  $h(A)=h(B) $,   there is an edge  $ A\leftarrow   B$ in~${\mathcal O}$. This condition is a  very weak form of a composition law  in~$\mathcal O$:  it may be rephrased by saying that if there is a finite chain of edges $A \leftarrow \cdots  \leftarrow B$ in~${\mathcal O}$ and $h(A)=h(B) $,  then there is an edge   $A\leftarrow  B$  in~$\mathcal O$. The reader may check   that in    our examples of phylogenetic quivers  all  vertices  are regular.
 
 We now state the main result of this section. 

\begin{theor}\label{thththt2++++eee} The clade  of any regular vertex    of   a small  phylogenetic quiver   is   a small phylogenetic quiver.
\end{theor}

The proof of this theorem occupies the rest of the section.

\subsection{Lemmas}\label{Two lemmasdd} Consider  a vertex~$A$  of   a  small phylogenetic quiver~${\mathcal O}$ and a vertex  $B\in {\mathcal O}_A$. Set $m=h(A)\geq 0$ and $n=h(B) $. By Theorem~\ref{ththdbdnht2++++eee}, we have $n\geq m$ and $h_A(B)\geq n-m$. Recall the isotypy classes $[A]\in {\mathcal O}_m$, $[B]\in {\mathcal O}_n$, and the iterated parental map $p^{n-m}:
{\mathcal O}_n \to  {\mathcal O}_m$.

\begin{lemma}\label{ththnewlemma23+eee} If  $p^{n-m}([B])=[A]$, then~$B$ is a phylogenetic  vertex  of  ${\mathcal O}_A$ and $h_A(B) = n-m$.
\end{lemma}

\begin{proof} Suppose first that $n=m$. Then the condition $p^{n-m}([B])=[A]$ means that $B\sim A$. Then~$B$ is primitive in ${\mathcal O}_A$ by Lemma~\ref{thththt2++++rreee}. Therefore $h_A(B) = 0= n-m$ and  $B$ is phylogenetic  in ${\mathcal O}_A$. Suppose now that $n>m$. Pick  universal evolutions~$\alpha$ for~$A$ and  $\beta=(B_0 \leftarrow \cdots  \leftarrow  B_n=B)$ for~$B $ in~$\mathcal O$.
The condition $p^{n-m}([B])=[A]$ implies that $B_m \sim A$. Then   the terminal segment $\beta^m=( B_m \leftarrow \cdots  \leftarrow  B_n  )$ of~$\beta$ is a full evolution in ${\mathcal O}_A$ for~$B$ of length $n-m$.    For any evolution~$\gamma$ from~$A$ to~$B$   in~$\mathcal O$,  the  universal evolution~$\beta $  must embed  in $ \alpha \gamma$. Comparing the heights of vertices   in~$\mathcal O$, we observe that such an embedding carries $\beta^m$   to~$\gamma$ and yields an embedding of $\beta^m$   in~$\gamma$. Thus, $\beta^m$ is a universal evolution for~$B$ in $ {\mathcal O}_A$. So, $B$  is a phylogenetic  vertex  of ${\mathcal O}_A$  and $h_A(B) = n-m $. \end{proof}

\begin{lemma}\label{ththnewlemma23+eeeonemore} If   $p^{n-m}( [B]) \neq [A]$ and $A$ is regular, then~$B$ is a phylogenetic  vertex  of ${\mathcal O}_A$ and $h_A(B) = n-m+1$.
\end{lemma}

\begin{proof} Suppose first that $n=m$. Then the condition $p^{n-m}([B]) \neq  [A]$ means that~$  A$ and $B$ are not isotypic in~${\mathcal O}$.   By Lemma~\ref{thththt2++++rreee},  $B$ is not a primitive vertex  of ${\mathcal O}_A$ and so $h_A(B) \geq 1$.    Since $B\in {\mathcal O}_A$, there is an evolution from~$A$ to~$B$ in~$\mathcal O$ of non-zero  length. The regularity of $A$ and the assumption  $h(B)=n=m=h(A)$  imply that there is an edge  $ A \leftarrow  B$ in~${\mathcal O}$. This edge  yields a universal evolution for~$B$ in ${\mathcal O}_A$ of  length~1. Thus,~$B$ is a phylogenetic vertex  of ${\mathcal O}_A$ and  $h_A(B)  = 1=n-m+1$. 

 Suppose now that $n>m$. Pick a universal evolution $\beta=(B_0 \leftarrow \cdots  \leftarrow  B_n)$ for $B=B_n$ in~${\mathcal O}$ and an arbitrary   evolution $\gamma=(A_0 \leftarrow   \cdots  \leftarrow A_k)$   from $A=A_0$ to $B=A_k$  in~${\mathcal O}$ (for some $k\geq 0$). 
      The same argument as in the proof of the previous lemma shows that the terminal segment $\beta^m=(B_m \leftarrow \cdots  \leftarrow  B_n)$ of~$\beta$ embeds in~$\gamma$. In particular,  $B_m \sim A_i$ for some~$i $. So $A\leq A_i \leq B_m$ and   $B_m\in {\mathcal O}_A$. The condition $p^{n-m}([B]) \neq [A]$ implies that $B_m$ and  $  A=A_0$  are not isotypic, i.e., $i \neq 0$. Therefore there is an evolution from $A=A_0$ to $B_m\sim A_i$ in~$\mathcal O$ of  length $\geq i \neq 0$.  The regularity of $A$ and the equalities $h(B_m)=m=h(A)$  imply that there is an edge  $A \leftarrow  B_m $ in~${\mathcal O}$. Concatenating this edge  with   $\beta^m $, we obtain an evolution $\beta^m_+ $ from~$A$ to $B=B_n$ of
length $n-m+1$. By the above,  $\beta^m_+ $ embeds in~$\gamma$. Therefore, $\beta^m_+ $ is a universal evolution for $B\in {\mathcal O}_A$, $B$ is phylogenetic in ${\mathcal O}_A$, and   $h_A(B) = n-m+1$.  \end{proof}



\subsection{Proof of Theorem~\ref{thththt2++++eee}}\label{Two lemmasddvnvn} Let $A$ be a regular vertex    of   a  phylogenetic quiver~${\mathcal O}$. Since~$\mathcal O$ is small, so is the clade ${\mathcal O}_A$ of~$A$. In view of Lemmas~\ref{ththnewlemma23+eee} and~\ref{ththnewlemma23+eeeonemore}, we need only to prove that ${\mathcal O}_A$ is monotonous, i.e.,    that $h_A(B)\leq h_A(C)$ for any $B, C\in  {\mathcal O}_A$ such that there is an edge  $B \leftarrow C$ in~${\mathcal O}$. Set $m=h(A), n= h(B)$ and $k=h(C)$. By the monotonicity of~${\mathcal O}$ and the definition of the height,    $$m\leq n\leq k\leq n +1.$$ We distinguish three cases.

Case $m=n=k$. If $A\sim B$, then $h_A(B)=0 \leq h_A(C)$.  If $A \sim C$, then $A \leq B \leq C \leq A$ so that $A \sim B$ and we proceed as above. If neither~$B$ nor~$C$ are isotypic to $A$, then Lemma~\ref{ththnewlemma23+eeeonemore} applies to $B,C$ and gives $h_A(B)=1= h_A(C)$.

Case $m< n=k$. Since there is an edge  $B \leftarrow C$ in~$\mathcal O$ and $h(B)=n=k=h(C)$, Theorem~\ref{aboutpreceq}(iii) implies that $p([B])=p([C])$. Therefore $p^{n-m}([B])=p^{n-m}([C])$.  Then either $p^{n-m}([B])=[A]$ and Lemma~\ref{ththnewlemma23+eee}   gives
$$h_A(B)= n-m=k-m=h_A(C)$$
or $p^{n-m}([B])\neq [A]$ and  Lemma~\ref{ththnewlemma23+eeeonemore} gives
$$h_A(B)= n-m+1= k-m+1=h_A(C). $$

Case $  k=n+1$. Concatenating a  universal evolution for~$B$ in~$\mathcal O$ with the edge  $B \leftarrow C$ we obtain  
a  short full evolution for~$C$ in~$\mathcal O$. The latter evolution is universal  and so $p([C])= [B] $. Therefore $p^{n-m}([B])=p^{n-m+1}([C])$.  If $p^{n-m}([B])=[A]$, then    Lemma~\ref{ththnewlemma23+eee}   gives  
$$h_A(B)= n-m \,\, \,  {\text {and}}\,\, \, h_A(C)=k-m. $$
If $p^{n-m}([B])\neq [A]$, then  Lemma~\ref{ththnewlemma23+eeeonemore}  gives
$$h_A(B)= n-m+1 \,\, \,  {\text {and}}\,\, \, h_A(C)=k-m+1. $$
In both cases, $h_A(B) <  h_A(C)$.

 \section{The quiver  of finite ultrametric spaces}\label{sect35df}

We  form a phylogenetic quiver  from   finite ultrametric spaces. We first recall  the definition  of   an ultrametric space.

\subsection{Ultrametrics and contractions}\label{sect30}\label{Contractions--}\label{Contractions}  A   \emph{metric space}\index{space!metric} is  a pair $(X,d)$ consisting of  a non-empty set~$X$ and a map $d:X \times X \to \RR_+=[0, \infty) $,  the \emph{metric}, such that for all $x,y,z\in X$, we have $d(x,y)=d(y,x) $, $d(x,y)= 0 \Longleftrightarrow x=y$, and
\begin{equation}\label{trian}d(x,y) \leq  d(x,z) + d(y,z).\end{equation}
      An isometry between metric spaces is a metric-preserving   bijection. A metric space $(X,d)$ is    \emph{finite} if~$X$ is a finite set.

 An   \emph{ultrametric space}  is a metric\index{space!ultrametric} space $(X,d)$    such that for all $x,y,z\in X$,
\begin{equation}\label{ult} d(x,y) \leq \text{max}(d(x,z), d(y,z)).  \end{equation}
   The map~$d$ is called then an   \emph{ultrametric}.
 The condition \eqref{ult} is stronger than  \eqref{trian}; it implies that for any   $x,y,z\in X$, two of the numbers $d(x,y)$, $ d(x,z)$, $d(y,z)$ are equal   and are greater than or equal to the third number. 

  For a real number $\epsilon> 0$, we call  a map $f:X\to Y$ between   metric spaces $X=(X,d)$ and $Y=(Y,\rho)$  an \emph{$\epsilon $-contraction}
if    $f(X)=Y$  and
\begin{equation}\label{contt} \rho(f(x),f(y))  =d(x,y)- \varepsilon   \end{equation}
for any distinct $x,y \in X$. Various $\epsilon $-contractions with   $  \epsilon \in (0,\infty) $ are collectively called \emph{contractions}.
Contractions are surjective but not necessarily bijective. The composition of  two contractions is not necessarily a contraction.

\begin{theor}\label{th111} Let $\mathcal{U}$ be the quiver  whose vertices  are finite   ultrametric spaces and whose edges are
  contractions and isometries.
 Then:

 (i) Two finite ultrametric spaces are isotypic  to each other in~$\mathcal{U}$   if and only if they are isometric;

 (ii) A finite ultrametric space   is a primitive vertex   of~$\mathcal{U}$   if and only if it consists of a single point;

 (iii) The height of a finite  ultrametric space $X=(X,d)$ in~$\mathcal U$ is equal to the number of non-zero elements in the set $d(X\times X) \subset \RR_+$;

  (iv) The quiver~$\mathcal{U}$ is    phylogenetic;

  (v)  All vertices  of~$\mathcal{U}$ are regular.
  \end{theor}

  The proof of Theorem~\ref{th111} occupies the rest of the  section.
We begin with notation and a lemma.   For a  finite metric space $ X=(X,d)$,   set
   $$ \vert \vert X\vert \vert =\sum_{x,y\in X} d(x,y)=   \sum_{x,y\in X, x \neq y} d(x,y).$$ 
   If $\card(X)\geq 2$,  then   set
  $$ \vert X\vert = \text{min} \{ d(x,y) \, \vert \, x,y \in X, x\neq y\} .  $$
   
    \begin{lemma}\label{contr} Let    $f: X\to Y$   be an   $\epsilon $-contraction between finite metric spaces $X=(X,d)$ and $Y=(Y,\rho)$ where $\varepsilon >0$. If   $\card(Y)\geq 2$, then $\vert \vert Y\vert \vert <  \vert \vert X\vert \vert$. If~$f$ is not a bijection, then $\card(X)\geq 2$ and   $ \vert X\vert = \varepsilon $.
\end{lemma}

\begin{proof}  Since $f(X)=Y$, we can pick for each $a\in Y$   a point $\overline a \in f^{-1}(a) $. Then
$$\vert \vert Y\vert \vert=\sum_{a,b\in Y, a \neq b } \rho (a,b) =\sum_{a,b\in Y, a\neq b} (d(\overline a , \overline b)-\varepsilon ) < \sum_{a,b\in Y, a\neq b } d(\overline a , \overline b) \leq  \vert \vert X\vert \vert.$$  

  If   $f$ is non-injective,  then $\card(X) \geq \card(Y)+ 1  \geq 2$ and $f(a)=f(b)$ for some  distinct $a,b\in X$.  Formula~\eqref{contt} implies that $d(a,b)=\varepsilon$  and   $d(x,y) \geq \varepsilon$ for any distinct $x,y\in X$. Hence,    
$ \vert X\vert = \varepsilon $. \end{proof}

 \subsection{Proof of Theorem~\ref{th111}}\label{sectnumerinvsfgh}   Since isometries   are   edges in~$\mathcal U$ and each edge  determines a length 1 evolution,  isometric   ultrametric spaces  are isotypic in~$\mathcal{U}$.  Conversely,  consider isotypic  $X,Y \in \mathcal{U}$. If $\card(X)=1$, then the condition $Y \leq X$ implies that $\card(Y)=1$.  If $\card(Y)=1$, then the condition $X \leq Y$ implies that $\card(X)=1$. In both cases, $X$ is isometric to~$Y$. Suppose that $\card(X)\geq 2$ and $\card(Y)\geq 2$. Since $Y\leq X$, there is an evolution $Y\leftarrow \cdots \leftarrow X$  in~$\mathcal{U}$. By Lemma~\ref{contr}, either      all edges in this evolution are isometries or $\vert \vert Y\vert \vert < \vert \vert X\vert \vert$. Similarly,   the relation $X\leq Y$  implies that either $X$,~$Y$ are  isometric or~$\vert \vert X\vert \vert < \vert \vert Y\vert \vert$. Since we cannot   have  $\vert \vert Y\vert \vert < \vert \vert X\vert \vert < \vert \vert Y\vert \vert$, the only   option is that~$X$  and~$Y$ are isometric. This implies  (i).

  To proceed, we define certain edges in~$\mathcal U$. Given $   \epsilon \in \RR  $ and a finite ultrametric space $ X=(X,d)$ having  at least two points,     the formula
 \begin{equation*}
    d_\varepsilon (x,y)=\left\{
                \begin{array}{ll}
                0 \quad {\text {if}} \,\,  x=y , \\
                  d(x,y)-\varepsilon \quad {\text {if}} \,\,   x \neq y   
                \end{array}
              \right.
\end{equation*} defines a map $d_\varepsilon:X\times X \to \RR$. If $  \epsilon < \vert X\vert$, then $d_\varepsilon$ is an ultrametric  in $X$ and the identity map $\id_X:(X,d)\to (X, d_\varepsilon)$ is   a  bijective $\varepsilon$-contraction. For   $  \epsilon = \vert X\vert$, the   map $d^\bullet=d_{\vert X\vert}:X\times X \to \RR$
satisfies all requirements   on an ultrametric   except one:
there are distinct   $x,y\in X$ with $d^\bullet (x,y)=0$.  We define a relation $\sim_{d^\bullet}$ in~$X$   by $x_1\sim_{d^\bullet} x_2$ if $d^\bullet (x_1,x_2)=0$. It is straightforward to check that $\sim_{d^\bullet}$ is an equivalence relation. Let $Y=X/{\sim_{d^\bullet}}$ be the  quotient set and   let $p:X\to Y$ be the projection. Then there is a unique   map $\rho:Y\times Y\to \RR $ such that $ {d^\bullet}(x,y)=\rho(p(x), p(y))$ for all $x,y\in X$. The map~$\rho$  is an ultrametric. We   denote the ultrametric space  $(Y,\rho)$ by   $u(X)$. Clearly, the projection $p:X\to u(X)$ is a non-injective  $\vert X\vert$-contraction. Applying this  construction recursively, we obtain an   evolution  in~$\mathcal U$
 \begin{equation}\label{evol47}  u^m(X) \leftarrow   \cdots \leftarrow u^2(X) \leftarrow u(X) \leftarrow X \end{equation}
 where~$m $ is the smallest integer such that $u^m(X)$ has only one point. Thus,   $ X  $ has a 1-point  ancestor.
We can now prove Claim (ii) of the theorem. By the definition of~$\mathcal U$, the only edges from a 1-point ultrametric space~$A$ to    vertices  of~$\mathcal U$ are isometries.  Thus, all ancestors of~$A$ are isometric to~$A$ and~$A   $ is primitive. If $X \in  \mathcal{U} $ is   primitive, then~$X$ is  isotypic to all its ancestors and, in particular, is isotypic to a 1-point   space. By (i),  $X$ is   a 1-point space. This proves (ii).

  We verify now that every  vertex  $X\in \mathcal{U}$
  is phylogenetic. Let~$\alpha$ be the full evolution \eqref{evol47} for~$X$. We claim that~$\alpha$   is universal.
  We must show that~$\alpha$ embeds in  an arbitrary full evolution  for~$X$, say,
  $$ \beta=( B_0  \stackrel{f_1} {\longleftarrow}  B_1  \stackrel{f_2} {\longleftarrow} \cdots  \stackrel{f_{n-1}} {\longleftarrow} B_{n-1}  \stackrel{f_n} {\longleftarrow} B_n=X) $$
  where $B_0$ is a 1-point space.
  Note   that if for some $k= 1,..., n-1$, the edge  $f_k:B_{k} \to B_{k-1}$   in~$\beta$ is an isometry, then we can delete   $B_k$ from $\beta$ and replace    $ f_k, f_{k+1}$ with their  composition  $f_k f_{k+1}: B_{k+1}\to B_{k-1}$ which is   a contraction or an isometry depending on whether $   f_{k+1}$   is   a contraction or an isometry.   This gives a shorter  full  evolution $\beta'$ for~$X$ which embeds in~$\beta$. It suffices to prove that~$\alpha$ embeds in $\beta'$. Similarly, if the edge  $f_n:B_{n} \to B_{n-1}$   in~$\beta$ is an isometry, then we can delete   $B_{n-1}$ from $\beta$ and replace   $ f_{n-1}, f_{n}$ with their composition. Thus, we can reduce ourselves to the case where all edges  in $\beta$ are contractions. In the same way, we can get rid of \emph{bijective} contractions in~$\beta$ using the obvious  fact that the composition of a bijective $\varepsilon$-contraction with any $\varepsilon'$-contraction is an $(\varepsilon+\varepsilon')$-contraction.
  It remains therefore to treat the case where all edges  in $\beta$ are non-bijective contractions.
  In particular,   $f_n:X=B_n\to B_{n-1}$ is a non-bijective $\varepsilon$-contraction for some $\varepsilon>0$. By Lemma~\ref{contr},  $ \varepsilon=\vert X\vert $. Consequently,  $B_{n-1}$ is isometric to $u(X)$. Proceeding by induction, we   obtain that $B_{n -k}$ is isometric to $u^k(X)$    for all $k\geq 1$. Since $m$   is the smallest integer such that $u^m(X)$ has only one point and since $B_0$ is the only 1-point space in the evolution~$\beta$, we   conclude that $m=n$ and that the evolutions $\alpha$ and $\beta$ are isotypic.   In particular,~$\alpha$ embeds in~$\beta$. This proves   the universality of~$\alpha$.

 To prove (iii),   let $N( X)  $  be the number of non-zero elements in  $d(X\times X) \subset \RR_+$.      By the above, the  evolution \eqref{evol47} is universal and so $h(X)=m$. The values of the ultrametric in $u(X)$  are obtained from those of~$d$ by  finding  the smallest non-zero value of~$d$ and subtracting  it from all   non-zero values of~$d$. Therefore $ N( u(X))= N( X)-1$. Inductively,
   $ N( u^k(X))= N( X)-k$ for   $k=1,..., m$. Since $u^m(X)$ is a 1-point set,
$0= N( u^m(X))= N( X)-m$. Thus, $h(X)=m=N( X) $.

We now prove (iv). That~$\mathcal U$ is small follows from (i).  As we know, all vertices  of~$\mathcal U$ are phylogenetic.  For any contraction or isometry $X \to Y$ between finite ultrametric spaces, one easily sees that $N(X) \geq N(Y)$. Therefore $h(X)=N(X) \geq N(Y)=h(Y)$.
Thus,  $\mathcal U$ is monotonous.

We leave it to the reader to check that all vertices  of~$\mathcal{U}$ are regular.

 \section{The quiver  of  finite metric spaces}\label{sect35dfmetrc}

We  form a phylogenetic quiver  from   finite  metric spaces.  We start by   defining   trim metric spaces and drifts following \cite{Tu1},  \cite{Tu2}.

\subsection{Trim metric spaces}\label{sect30ofofofo}
A  metric space $(X,d)$ is   \emph{trim} if  either $\card (X)= 1$ or    for each $x\in X$, there are distinct   $y,z\in X\setminus \{x\}$ such that $$d(x,y)+d(x,z)= d(y,z).$$ The latter equality may be  expressed by saying that~$x$ lies between~$y$ and $z$.  The class of
finite  trim metric spaces is quite narrow.
In particular, there are no trim   metric spaces  having  two   or three points.
 A
finite subset  of a Euclidean space   with   $  \geq 2$ points and with the induced metric  cannot be trim:  such a subset   must  contain   a pair of points lying at the maximal distance;   these points cannot lie between   other  points of the subset.   For examples of   trim metric spaces, see \cite{Tu1}, \cite{Tu2}.

With any metric space $(X,d)$, we associate a function $\underline d: X \to \RR_+ $ as follows:
 if~$X$ has  only one  point, then $\underline d =0$; if $X$ has two  points $x,y$, then $\underline d(x)=\underline d(y)= d(x,y)/2$; if~$X$ has three or more points, then for all $x \in X$,
$$ {\underline d} (x)= \inf_{y,z \in X\setminus \{x\}, y \neq z }   \frac{ d(x, y)+ d(x, z)-d(y,z)}{2}   \geq 0 . $$
It is easy to check that $\underline d (x) +   \underline d(y) \leq d(x,y)$ for any distinct $x,y\in X$, see \cite{Tu1}. It follows from the definitions  that $\underline d=0$ if and only if    $(X,d)$ is  trim.

\subsection{Drifts}
 We call  a map $f:X\to Y$ between   metric spaces $X=(X,d)$ and $Y=(Y,\rho)$  a  \emph{drift}
 if   $f(X)=Y$  and   for any distinct $x,y \in X$,
\begin{equation}\label{conttrrrrxcx} \rho(f(x),f(y))  =d(x,y)-   \underline d (x) -   \underline d(y) .  \end{equation}
 A  drift is surjective but not necessarily bijective.  If $X$ is trim, then all  drifts $X \to Y$ are isometries.   We state a version of Theorem~\ref{th111} for metric spaces.

\begin{theor}\label{th111+++++++} Let $\mathcal{M}$ be the quiver  whose vertices  are finite    metric spaces and whose edges are
drifts and isometries.
 Then:

 (i) Two finite  metric spaces are isotypic  in~$\mathcal{M}$   if and only if they are isometric;

 (ii) A finite  metric space   is a primitive vertex   of~$\mathcal{M}$   if and only if it is trim;

 (iii)  The quiver  $\mathcal{M}$ is    phylogenetic;
 
 (iv)  All vertices  of~$\mathcal{M}$ are regular.
  \end{theor}

\begin{proof} Recall from Section~\ref{Contractions} the isometry invariant
  $ \vert \vert X\vert \vert  $ of   a finite metric space~$ X $. It is clear that for any  edge  $ f:X\to Y$ in~$\mathcal{M}$, either  $\vert \vert X\vert \vert > \vert \vert Y\vert \vert$ or~$f$ is an isometry. If  $X,Y \in \mathcal{M}$ are isotypic, then  applying this argument to the edges in an evolution $Y\leftarrow \cdots \leftarrow X$, we obtain that either  $\vert \vert X\vert \vert > \vert \vert Y\vert \vert$ or all these edges   are isometries. Similarly,     either $\vert \vert Y\vert \vert > \vert \vert X\vert \vert$ or all edges in an evolution $X\leftarrow \cdots \leftarrow Y$ are isometries.   This implies   (i).

  To proceed,   consider a  finite  metric space $ X=(X,d)$ with  at least two points. The formula
 \begin{equation*}
    d^\bullet (x,y)=\left\{
                \begin{array}{ll}
                0 \quad {\text {if}} \,\,  x=y , \\
                  d(x,y)-    \underline d (x) -   \underline d(y) \quad {\text {if}} \,\,   x \neq y   
                \end{array}
              \right.
\end{equation*} defines a map $d^\bullet:X\times X \to \RR$ which
satisfies all requirements   on a metric   except, possibly, one:
there may be distinct   $x,y\in X$ with $d^\bullet (x,y)=0$.  We define an equivalence relation $\sim_{d^\bullet}$ in~$X$   by $x_1\sim_{d^\bullet} x_2$ if $d^\bullet (x_1,x_2)=0$.   Let $Y=X/{\sim_{d^\bullet}}$ be the  quotient set and   let $p:X\to Y$ be the projection. Then there is a unique   map $\rho:Y\times Y\to \RR $ such that $ {d^\bullet}(x,y)=\rho(p(x), p(y))$ for all $x,y\in X$. The map~$\rho$  is a metric in~$Y$, and we   denote the  metric space  $(Y,\rho)$ by   $v(X)$. Clearly, the projection $ X\to v(X)$ is a drift.   Applying this  procedure recursively, we obtain an   evolution  in~$\mathcal M$
 \begin{equation}\label{evol47ff}  v^m(X) \leftarrow   \cdots \leftarrow v^2(X) \leftarrow v(X) \leftarrow X \end{equation}
 where $m\geq 0$ is the smallest integer such that $v^m(X)$ is trim. (The existence of such an~$m$ follows from the fact that if a  drift $X \to Y$ is bijective, then $Y$ is trim, cf.\ \cite{Tu1}, Lemma 2.1). We conclude that
  each vertex    of~$ \mathcal{M} $ has a trim   ancestor.

We can now prove Claim~(ii) of the theorem. By the definition of~$\mathcal M$, the only edges from a trim metric space~$X$ to   vertices  of~$\mathcal M$ are isometries.  Thus, all ancestors of~$X$ are isometric to~$X$, and so~$X   $ is primitive. Conversely, if $X \in  \mathcal{M} $ is   primitive, then~$X$ is  isotypic to all its ancestors. By the above,  $X  $ is isotypic to a trim metric   space. By (i),  $X$ is itself a trim metric space.

 As  in the proof of Theorem~\ref{th111},  the   evolution \eqref{evol47ff}  is     universal, and so all vertices  of~$\mathcal M$ are phylogenetic.
   We leave it to the reader to check   that~$\mathcal M$ is   monotonous and all its vertices  are regular.
\end{proof}

%
%

                     \end{document}